\documentclass[12pt]{amsart}
\usepackage{amssymb}
\usepackage{amsfonts}
\usepackage{graphicx}
\usepackage{amsmath}
\usepackage{xcolor}

\newtheorem{theorem}{Theorem}

\newtheorem{example}[theorem]{Example}

\newtheorem{lemma}[theorem]{Lemma}

\newcommand{\C}{\mathbb{C}}
\newcommand{\N}{\mathbb{N}}

\newcommand{\R}{\mathbb{R}}

\begin{document}

\title{The apolar inner product and a Bombieri type inequality for polynomials}
\author{ J. M. Aldaz, A. Bravo and H. Render}
\address{J.M. Aldaz: Instituto de Ciencias Matem\'aticas (CSIC-UAM-UC3M-UCM)
and Departamento de Matem\'aticas, Universidad Aut\'onoma de Madrid,
Cantoblanco 28049, Madrid, Spain.}
\email{jesus.munarriz@uam.es}

\address{A. Bravo: Instituto de Ciencias Matem\'aticas (CSIC-UAM-UC3M-UCM)
and Departamento de Matem\'aticas, Universidad Aut\'onoma de Madrid,
Cantoblanco 28049, Madrid, Spain.}
\email{ana.bravo@uam.es}

\address{H. Render: School of Mathematical Sciences, University College
Dublin, Dublin 4, Ireland.}
\email{hermann.render@ucd.ie}


\begin{abstract} We consider inequalities of Bombieri type for polynomials that need not be homogeneous, using the apolar inner product. 
\end{abstract}

\maketitle

\section{Introduction} The apolar inner product $\langle  \cdot,\cdot\rangle_{a}$ on spaces of polynomials originates within the Theory of Invariants, from the XIX century, and has since found several other applications. For instance, it is essential in the proof of the Fischer decomposition of polynomials, which in turn helps to find polynomial solutions to certain differential equations (in this regard, among the many references that could be given we mention  \cite{Shap89}, \cite{EbSh95}, \cite{EbSh96}, \cite{Rend08}, \cite{KL18}, \cite{AlRe23})). 

Related to the apolar inner product is the celebrated Bombieri's inequality for homogeneous polynomials $f$ and $g$ (cf. for instance \cite{BBEM90}, \cite{Beau92}, \cite{Beau96}, \cite{Beau97}, \cite{Etay21}, \cite{Pina12}, \cite{Zeil}) which for the apolar norm
  states that
$$\|f g\|_a \ge \|f\|_a \|g\|_a.
$$
By induction, the extension to the case of $n$ homogeneous polynomials is immediate:
\begin{equation}\label{multBombieri}
\|f_1 \cdots f_n\|_a \ge \|f_1\|_a \cdots \|f_n\|_a.
\end{equation}
 As is well known, the inequality $\|f g\|_a \ge \|f\|_a \|g\|_a$ still holds when only one of the polynomials is homogeneous, cf. \cite{Zeil} for instance, or the final part of the next section.
Hence, it is natural to explore what happens when both  $f$ and $g$ are non-homogeneous. This paper, which studies the preceding question, is organized as follows.
The relevant definitions and standard results appear in Section 2. A simple example given in Section 3, shows that Bombieri's inequality does not hold with constant 1 if both polynomials are non-homogeneous. But by reduction to the homogeneous case and an application of Bombieri's inequality, we show that letting $n_i$ be the degree of $f_i$, the inequality
	\[
[(n_1 +  \cdots + n_s)! ]^{1/2 }\|f_1 \cdots f_n\|_a \ge \|f_1\|_a \cdots \|f_n\|_a
\]		
 holds. It may well be that a direct proof will lead to better bounds, 
but we have not found such a proof from first principles. Section 4 presents some special cases where better bounds do hold. Finally, in Section 5 we consider powers of a single polynomial $P$. 

\section{Notation, definitions and background results}

Denote by $\mathbb{C}[z_{1},\dots,z_{d}]$ the vector space of all
polynomials  in  $z=\left(  z_{1},\dots,z_{d}\right)$ with complex coefficients, and by $\mathbb{C} [z_{1},\dots,z_{d}]_{\le n}$  the subspace of  $\mathbb{C}[z_{1},\dots,z_{d}]$ consisting of polynomials with degree bounded by $n$.
Let $P$ and $Q$ be polynomials of degree $N$ and $M$, respectively given by
\[
P\left(  z \right)  =\sum_{\alpha\in\mathbb{N}^{d},\left\vert
\alpha\right\vert \leq N}c_{\alpha}z ^{\alpha}\text{ and }Q\left(  z\right)
=\sum_{\alpha\in\mathbb{N}^{d},\left\vert \alpha\right\vert \leq
M}d_{\alpha}z^{\alpha},
\]
where the standard notation for multi-indices is used: for $\alpha=\left(  \alpha
_{1},\dots ,\alpha_{d}\right)  \in\mathbb{N}^{d}
$, we write   $z^{\alpha
}=z_{1}^{\alpha_{1}}\cdots  z_{d}^{\alpha_{d}},$  $\alpha!=\alpha_{1}
! \cdots \alpha_{d}!$, and $\left\vert \alpha\right\vert =\alpha_{1}+\cdots
+\alpha_{d}$. We call the largest value of $\left\vert \alpha\right\vert $ such that $c_\alpha \ne 0$ the total degree of $P$, or
just its degree.

Given $P\left(  z\right)  $, we denote by
$P^{\ast}\left(  z\right)  $ the polynomial obtained from $P\left(  z\right)
$ by conjugating its coefficients, and by $P\left(  D\right)  $  the linear
differential operator obtained from $P\left(  z\right)
$ by replacing each variable $z_{j}$ with the
differential operator $\frac{\partial}{\partial z_{j}}$, for $j = 1, \dots ,d$.
The \emph{apolar inner product} $\langle  \cdot,\cdot\rangle_{a}$ on
$\mathbb{C}[z_{1},\dots,z_{d}]$ is defined by
\begin{equation}\label{apolar}
\left\langle P,Q\right\rangle_{a} :=\left( Q^*\left(
D\right)  P\right) \;(0)=\sum_{\alpha\in\mathbb{N}^{d}}\alpha!c_{\alpha
}\overline{d_{\alpha}},
\end{equation}
and the associated {\em apolar norm}, by
$
\left\Vert f\right\Vert _{a}=\sqrt{\left\langle f,f\right\rangle_{a}}.
$
Note that for monomials $z^\alpha$ and $z^\beta$, definition (\ref{apolar}) means that $\left\langle z^\alpha, z^\beta \right\rangle_{a} = 0$ if $\alpha \ne \beta$, and $\left\langle z^\alpha, z^\alpha \right\rangle_{a} = \alpha !$. Furthermore, from these relations definition (\ref{apolar}) is recovered by sesquilinearity.

Bombieri's inner product $\left\langle \cdot, \cdot \right\rangle_{b}$ is the apolar inner product  with a different normalization: $\left\langle z^\alpha, z^\alpha \right\rangle_{b} = \alpha !/|\alpha|!$. Thus, $\left\langle \cdot, \cdot \right\rangle_{b}$ depends on the total degree $|\alpha|$ of the monomials, or more generally, the total degree of the homogeneous polynomials.

 It is well known and easy to see that Bombieri's inequality extends to the case were one of the polynomials is arbitrary and the other is homogeneous. Let $f_m$ be homogeneous of total degree $m$ and let $P$ be a non-zero polynomial of total degree $k$, say 
\begin{equation*}
P\left( z \right) =P_{k}\left( z \right) +\cdots +P_{0}\left( z \right),
\end{equation*}
where for each $j=0, \dots ,k,$
 $P_{j}\left( z\right) $ is a homogeneous polynomial
of total degree $j$.  We call $P_k$ the principal part of $P$.
By the Pythagorean Theorem and Bombieri's inequality for the homogeneous case we have 
\[
\left\Vert P \cdot f_{m} \right\Vert _{a}^{2}
=
\left\Vert \sum_{j = 0}^k P_j\cdot f_m \right\Vert _{a}^{2}
=
\sum_{j = 0}^k  \left\Vert P_j\cdot f_m \right\Vert _{a}^{2}
\geq
\sum_{j = 0}^k  \left\Vert P_j \right\Vert _{a}^{2}  \left\Vert f_m \right\Vert _{a}^{2}
=
\left\Vert P \right\Vert _{a}^{2}  \left\Vert f_m \right\Vert _{a}^{2}.
\]

\section{A Bombieri type theorem}

In general, Bombieri's inequality  with constant 1 can fail: consider for instance 
$$
(x - 1) ( x + 1) = x^2 - 1.
$$
 Then $\|(x - 1)\|_a^2 \| ( x + 1)\|_a^2 = 4$, while  $\|x^2 - 1\|_a^2 = 3$.
 
 \vskip .2 cm
 
 Thus, Bombieri's inequality does not extend to the general case without some modification.

\begin{theorem} \label{BombieriType}
	Let $P^{(1)} ,\dots ,  P^{(s)} \in \mathbb{C}[z_{1},\dots,z_{d}]$ be nonzero 
	polynomials of   total degrees $n_1, \dots , n_s$ respectively.  Then
	\[
(n_1 +  \cdots + n_s)! \left\Vert P^{(1)} \cdots   P^{(s)} \right\Vert _{a}^{2}
\ge
\left\Vert P^{(1)} \right\Vert _{a}^{2}   \cdots \left\Vert P^{(s)} \right\Vert _{a}^{2}.
\]		
\end{theorem}

In order to prove the result, we introduce the following notation. Given $P$ of  degree $k$, say
\begin{equation*}
P\left( z \right) =P_{k}\left( z \right) +\cdots +P_{0}\left( z \right),
\end{equation*}
we select a variable $w$ different from $z_1, \dots, z_d$, and define the {\em one variable homogenization} of $P$ as
\begin{equation*}
P_{hom} \left( z , w \right) = P_{k}\left( z \right) + w P_{k -1}\left( z \right) + \cdots + w^k P_{0}\left( z \right).
\end{equation*}
Thus, $P_{hom}  \in \mathbb{C}[z_{1},\dots,z_{d}, w]$ is a homogeneous polynomial of degree $k$. The abbreviation $\left( z , w \right)$ obviously stands for the vector $\left( z_1, \dots, z_d , w \right)$.

The next Lemma states that the one variable homogenization commutes with taking products.

\begin{lemma}
\label{commute} Let $P^{(1)} ,\dots ,  P^{(s)} \in   \mathbb{C}[z_{1},\dots,z_{d}]$. Then 
	$
	P^{(1)}_{hom} \cdots P^{(s)}_{hom} = \left(P^{(1)} \cdots P^{(s)}\right)_{hom}.
	$
	\end{lemma}

\begin{proof} By induction it is enough to prove the case
$s = 2$.
Let $P , Q$ be polynomials with degrees $m$ and $n$ respectively, say
\[
P\left(  z \right)  =\sum_{\alpha\in\mathbb{N}^{d},\left\vert
\alpha\right\vert \leq m}c_{\alpha}z ^{\alpha}\text{ and }Q\left(  z\right)
=\sum_{\alpha\in\mathbb{N}^{d},\left\vert \alpha\right\vert \leq
n}d_{\alpha}z^{\alpha}.
\]
Select $\gamma \in \N^d$ with
$0 \le |\gamma | \le m + n$, and let $\ell_\gamma z^\gamma$ be the term of $PQ$ corresponding to $\gamma$.   Let $\le$ also denote the standard partial order on $\N^d$, that is, $\alpha \le \beta$ if and only if for all coordinates (so $1 \le i \le d$) we have $\alpha_i \le \beta_i$. Then 
\begin{equation}
\ell_\gamma z^\gamma 
=
\sum_{\alpha \le \gamma} c_{\alpha} z^\alpha d_{\gamma - \alpha} z^{\gamma - \alpha}.
\end{equation}
With the natural notation, the term of $(PQ)_{hom}$ corresponding to the vector $(\gamma, m + n - |\gamma|) \in \N^{d + 1}$ is thus
\begin{equation}
\ell_\gamma z^\gamma w^{m + n - |\gamma|}
=
\sum_{\alpha \le \gamma} c_{\alpha} z^\alpha d_{\gamma - \alpha} z^{\gamma - \alpha} w^{m + n - |\gamma|} 
= 
\sum_{\alpha \le \gamma} c_{\alpha} z^\alpha w^{m -  |\alpha|} d_{\gamma - \alpha} z^{\gamma - \alpha} w^{n -  |\gamma| +  |\alpha|},
\end{equation}
which coincides with the term of $P_{hom} Q_{hom}$ corresponding to the vector $(\gamma, m + n - |\gamma|)$.
\end{proof}

\begin{lemma}
\label{homogenization} Let $P  \in \mathbb{C}[z_{1},\dots,z_{d}]$ be a
	polynomial of total degree $|\alpha|$. Then 
	$$
	\|P\|_a^2 \le \|P_{hom}\|_a^2 \le   |\alpha| ! \|P\|_a^2.
	$$
	\end{lemma}

\begin{proof} The first inequality is obvious, since all terms of $P_{hom}$ have degree at least as large as those of $P$
and the coefficients are the same. For the second inequality, writing $P =P_{k} +\cdots +P_{0}$ (where each $P_j$ is homogeneous of degree $j$, and $P_k \ne 0$, so the total degree is $|\alpha | = k$) we have 
\begin{equation*}
\|P_{hom}\|_a^2
=
\left\langle P_{hom} ,P_{hom} \right\rangle_{a}   
=  
\left\langle \sum_{j = 0}^k w^{k - j} P_j, \sum_{j = 0}^k w^{k - j} P_j \right\rangle_{a}
=  
\sum_{j = 0}^k \left\langle  w^{k - j} P_j,  w^{k - j} P_j \right\rangle_{a}
\end{equation*}
\begin{equation*}
=  
\sum_{j = 0}^k (k - j)!\left\langle   P_j,  P_j \right\rangle_{a}
\le
k! \sum_{j = 0}^k \left\langle   P_j,  P_j \right\rangle_{a} = k! \|P\|_a^2.
\end{equation*}
\end{proof}

{\em Proof of Theorem \ref{BombieriType}.}  Let $P^{(1)} ,\dots ,  P^{(s)}  \in \mathbb{C}[z_{1},\dots,z_{d}]$ be nonzero 
	polynomials of  total degrees $n_1, \dots , n_s$ respectively.
  Then by the preceding lemmas and Bombieri's inequality (\ref{multBombieri}) for several polynomials, we have
	$$
	(n_1 +  \cdots + n_s)! \left\Vert P^{(1)} \cdots   P^{(s)} \right\Vert _{a}^{2}
\ge
\left\Vert \left(P^{(1)} \cdots P^{(s)}\right)_{hom} \right\Vert _{a}^{2}
$$
$$
=
\left\Vert P^{(1)}_{hom} \cdots P^{(s)}_{hom} \right\Vert _{a}^{2}
\ge
\left\Vert P^{(1)}_{hom} \right\Vert _{a}^{2} \cdots \left\Vert P^{(s)}_{hom} \right\Vert _{a}^{2}
\ge
\left\Vert P^{(1)} \right\Vert _{a}^{2} \cdots \left\Vert P^{(s)} \right\Vert _{a}^{2}.
$$
\qed		

Since the large constants come from the homogenization of the polynomials, having the additional information that the coefficients of the low order terms are zero yields the following theorem, with the same proof.

\begin{theorem} \label{BombieriTypebis}
	Let $P , Q\in \mathbb{C}[z_{1},\dots,z_{d}]$ be nonzero 
	polynomials of   total degrees $m$ and $n$ respectively, such that for some $0\le j \le m$ and some $0\le i \le n$, we have   $P\left( z \right) =P_{m}\left( z \right) +\cdots +P_{m -j}\left( z \right)$ and
$Q\left( z \right) = Q_{n}\left( z \right) +\cdots +Q_{n-i}\left( z \right)$.
Then
\[
(j + i)! \left\Vert P \cdot   Q \right\Vert _{a}^{2}
\ge
\left\Vert P \right\Vert _{a}^{2}   \left\Vert Q \right\Vert _{a}^{2}.
\]		
\end{theorem}

\section{Obtaining better bounds in some special cases}

Next we present some immediate improvements for certain special cases. If all the coefficients of $P$ and of $Q$ are non-negative, then cancellation cannot occur and thus Bombieri's inequality holds.

\begin{theorem} \label{BombieriPos}
	Let $P , Q  \in \mathbb{C}[z_{1},\dots,z_{d}]$ be nonzero 
	polynomials with real, non-negative coefficients. Then
	\[
\left\Vert P \cdot Q \right\Vert _{a}
\ge
\left\Vert P \right\Vert _{a}  \left\Vert Q \right\Vert _{a}.
\]		
\end{theorem}

\begin{proof} Let $P$ be a polynomial of total degree $m$, say 
\begin{equation*}
P\left( z \right) =P_{m}\left( z \right) +\cdots +P_{0}\left( z \right),
\end{equation*}
where for each $j=0, \dots ,m,$
 $P_{j}\left( z\right) $ is a homogeneous polynomial
of total degree $j$, and likewise, write
\begin{equation*}
Q\left( z \right) = Q_{n} \left( z \right) +\cdots +Q_{0}\left( z \right).
\end{equation*}
Since the coefficients of $P$ and $Q$ are non-negative, it follows from the definition of the apolar inner product that for all $j,k,r,s$ we always have $\left\langle  P_j Q_k, P_r Q_s\right\rangle_{a} \ge 0$. Thus 
\begin{equation}
\left\langle P Q, P Q\right\rangle_{a} 
=
\left\langle \sum_{j=0}^m
\sum_{k = 0}^{n} P_j Q_k, \sum_{j=0}^m \sum_{k = 0}^{n} P_j Q_k\right\rangle_{a} 
=
\sum_{j=0}^m \sum_{k = 0}^{n} \sum_{r=0}^m
\sum_{s = 0}^{n} \left\langle  P_j Q_k, P_r Q_s\right\rangle_{a} 
\end{equation}
\begin{equation}
\ge
\sum_{j=0}^m \sum_{k = 0}^{n}  \left\langle  P_j Q_k, P_j Q_k\right\rangle_{a} 
=
\sum_{j=0}^m \sum_{k = 0}^{n}  \left\| P_j Q_k\right\|_{a}^2
\ge
\sum_{j=0}^m \sum_{k = 0}^{n}  \left\| P_j \right\|_{a}^2 \left\|Q_k\right\|_{a}^2
\end{equation}
\begin{equation}
=
\sum_{j=0}^m   \left\| P_j \right\|_{a}^2 \sum_{k = 0}^{n} \left\|Q_k\right\|_{a}^2
= 
\left\| P \right\|_{a}^2 \left\| Q \right\|_{a}^2.
\end{equation}
\end{proof}

In order to try to lower the constant $(m + n)!$ from Theorem \ref{BombieriType} when dealing with two polynomials, one can search for homogenizations that  raise the norms of the homogenized polynomials by a smaller amount.

 Given $P$ of  degree $k$, say
\begin{equation*}
P\left( z \right) =P_{k}\left( z \right) +\cdots +P_{0}\left( z \right),
\end{equation*}
we select  variables $w_1, \dots, w_k$ different from $z_1, \dots, z_d$ (but we do not assume that all the variables $w_1, \dots, w_k$ are different among themselves), and define a {\em many variables homogenization} of $P$ as
\begin{equation*}
P_{hom} \left( z , w_1, \dots, w_k \right) = P_{k}\left( z \right) + w_1 P_{k -1}\left( z \right) + w_1 w_2 P_{k -2}\left( z \right)  + \cdots +  w_1 \cdots w_k P_{0}\left( z \right).
\end{equation*}

Suppose all the variables appearing in $P$ and $Q$ are raised to an even degree. Then a two variables homogenization leads to an improvement of the constant in Theorem \ref{BombieriType}. Let us denote by $2 \N$ the set of  even natural numbers. 

\begin{theorem} \label{BombieriEven}
	Let $m, n \in 2\N \setminus \{0\}$, and let $P , Q \in \mathbb{C}[z_{1},\dots,z_{d}]$ be nonzero 
	polynomials of the form \[
P\left(  z \right)  =\sum_{\alpha\in (2 \N)^d,\left\vert
\alpha\right\vert \leq m}c_{\alpha}z ^{\alpha}\text{ and }Q\left(  z\right)
=\sum_{\alpha\in (2 \N)^d,\left\vert \alpha\right\vert \leq
n}d_{\alpha}z^{\alpha}.
\] 
 Then
	\[
[\left(m + n\right)/2]! \left\Vert P \cdot Q \right\Vert _{a}
\ge
\left\Vert P \right\Vert _{a}  \left\Vert Q \right\Vert _{a}.
\]		
\end{theorem}

\begin{proof} We use the following two variable homogenization: write $P$ as 
\begin{equation*}
P\left( z \right) =P_{m}\left( z \right) + \cdots +P_{2}  (z) +P_{0}\left( z \right),
\end{equation*}
where for each even $j=0, \dots ,m,$
 the polynomial
 $P_{j}\left( z\right) $ is homogeneous of degree $j$, and set 
\begin{equation*}
P_{hom}\left( z , w_1, w_2\right) =P_{m}\left( z \right) + \cdots + (w_1 w_2)^{(m - 2)/2 } P_{2}  (z) + (w_1 w_2)^{m/2 } P_{0}\left( z \right),
\end{equation*}
doing likewise with $Q$ and $PQ$. By arguing as in Lemma \ref{commute}, we see that this homogenization also commutes with products, so $P_{hom} Q_{hom} = (PQ)_{hom}$. And as in Lemma \ref{homogenization}, we see that
\begin{equation*}
\|P_{hom}\|_a^2
=
\left\langle P_{hom} ,P_{hom} \right\rangle_{a}   
=  
\left\langle \sum_{j = 0}^{m/2}  (w_1 w_2)^{m/2 - j} P_{2j},\sum_{j = 0}^{m/2}   (w_1 w_2)^{m/2 - j} P_{2j}\right\rangle_{a}
\end{equation*}
\begin{equation*}
=  
\sum_{j = 0}^{m/2} \left\langle  (w_1 w_2)^{m/2 - j} P_{2j},  (w_1 w_2)^{m/2 - j} P_{2j} \right\rangle_{a}
=  
\sum_{j = 0}^{m/2} ((m/2 - j)!)^2\left\langle   P_{2j}, P_{2j}\right\rangle_{a}
\end{equation*}
\begin{equation*}
\le
((m/2)!)^2 \sum_{j = 0}^{m/2} \left\langle   P_{2j},  P_{2j} \right\rangle_{a} = ((m/2)!)^2 \|P\|_a^2.
\end{equation*}
And now we argue as in the proof of Theorem \ref{BombieriType}.
\end{proof}

For some purposes, a well known integral representation of the apolar inner product due to  V. Bargmann, cf. \cite{Bar}, turns out to be useful  (see also \cite{NeSh66} and \cite{Pina12}).

\begin{theorem} \label{barg}
Let $P$ and $Q$ be polynomials in $d$ complex variables. Then
\begin{equation}
\left\langle P,Q\right\rangle _{a}=\frac{1}{\pi^{d}}\int_{\mathbb{R}^{d}}
\int_{\mathbb{R}^{d}}P\left(  x+iy\right)  \overline{Q\left(  x+iy\right)
}e^{-\left|  x\right|  ^{2}-\left|  y\right|  ^{2}}dxdy<\infty\label{Bargman},
\end{equation}
where $dxdy$ is Lebesgue measure on $\mathbb{R}^{2d}$.
\end{theorem}

The next equality result for the homogeneous case is due to B. Reznick, cf. \cite[Main Theorem]{Rezn93}. In fact, Reznick's result is a characterization. We present the ``if'' part and note that the ``only if'' direction does not hold
when both polynomials are non-homogeneous.

\begin{theorem} \label{BombieriJensen} Let $P, Q  \in \mathbb{C}[z_{1},\dots,z_{d}]$, and suppose there is a linear change of variables 
	 $T: \C^d \to\C$ which makes $P$ and $Q$ depend on disjoint sets of variables. If the absolute value of the Jacobian determinant of $T$ takes the constant value $1$, then  
	\[
\left\Vert P \cdot Q \right\Vert _{a}
=
\left\Vert P \right\Vert _{a}  \left\Vert Q \right\Vert _{a}.
\]		
\end{theorem}

\begin{proof} Use the change of variables formula  and the integral representation from Theorem  \ref{barg}, integrating first with respect to the variables that appear in $P$ after $T$ is applied, and then with respect to the remaining variables.
\end{proof}

\begin{example} It is possible to have equality for non-constant polynomials of one variable, so the ``only if''  part of Reznick's characterization (stating that equality occurs only if after a suitable transformation, the polynomials depend on dijoint sets of variables) does not extend to the non-homogeneous case. 

For $0 \le t \le 1$ define $f(t):= \left\Vert (-1 + x)  \cdot (t + x) \right\Vert _{a}^2$, so $f(0) = f(1) = 3$, and set  $g(t):= \left\Vert (-1 + x)  \right\Vert _{a}^2 \left\Vert (t + x) \right\Vert _{a}^2$, so $g(0) = 2$ and  $ g(1) = 4$. Since both $f$ and $g$ are continuous functions of $t$, there exists a $c \in (0,1)$ such that $f(c) = g(c) $.\end{example}

If $Q$ is obtained from $P$ by deleting some terms, then $
\left\Vert Q  \right\Vert _{a}
\le
\left\Vert P  \right\Vert _{a}
$ by orthogonality.
It would be useful to have some analogous result when taking products. More
precisely, it is natural to ask whether the following ``monotonicity property'' holds:
let $P , Q  \in \mathbb{C}[z_{1},\dots,z_{d}]$ be nonzero 
polynomials such that $P$ has degree $n$, and suppose $P_{n + 1}$ is homogeneous of degree $n + 1$. Is it true that 
\[
\left\Vert (P_{n + 1} + P) \cdot Q \right\Vert _{a}
\ge
\left\Vert  P \cdot Q \right\Vert _{a}?
\]
The answer is no, as the following example shows.
Let $P (x)= 1$, $P_1(x) = tx$ and $Q(x) = 1 - tx$. Then
\[
\left\Vert \left(  1 + tx \right)  \left(  1 - tx \right)  \right\Vert _{a}^{2}
=
\left\Vert 1 - t^{2} x^{2}\right\Vert _{a}^{2}
=
1 + 2 t^4,
\]
while
\[
\left\Vert   1\cdot  \left(  1 - tx \right)  \right\Vert _{a}^{2}
=
1 + t^2.
\]
Taking $0 < t < 1/2$ we get $
\left\Vert (P_{1} + P) \cdot Q \right\Vert _{a}
<
\left\Vert  P \cdot Q \right\Vert _{a}.$

\section{Iterated products of a polynomial}

It follows from Bombieri's  inequality (\ref{multBombieri}) for several polynomials that if $P$ is 
\emph{homogeneous}, then for	$s\geq1$, $s\in \N$, we have
$
\left\Vert P^{s}\right\Vert _{a}\geq\left\Vert P\right\Vert _{a}^{s}.
$
An immediate consequence of Bargmann's integral representation is that this inequality holds for $P$ arbitrary.

\begin{theorem} \label{BombieriJensen2}
	Let $P  \in \mathbb{C}[z_{1},\dots,z_{d}]$. Then for every $s\geq1$, $s\in \N$, we have
	\[
	\left\Vert P^{ s}  \right\Vert _{a}
	\ge
	\left\Vert P \right\Vert _{a}^{ s}.
	\]		
\end{theorem}

\begin{proof} Apply Jensen's inequality to the integral representation from Theorem  \ref{barg}, using the convex function $\varphi (t) := |t|^s$, to conclude that $
	\left\Vert P^{ s}  \right\Vert _{a}^2
	\ge
	\left\Vert P \right\Vert _{a}^{2 s}.$
\end{proof}

Note that if $
\left\Vert P\right\Vert _{a} < 1$ then $\lim_{s\to \infty} \left\Vert P \right\Vert _{a}^{s} = 0$.  We show next that $
\left\Vert P^{ s}  \right\Vert _{a}
$ behaves very differently, provided $P$ is not constant.  

\begin{theorem}
	Let $P$ be a polynomial of degree $k\geq1.$ Then
	\[
	\lim_{s\rightarrow\infty}\sqrt[s]{\left\Vert P^{s}\right\Vert _{a}}=\infty.
	\]
	
\end{theorem}

\begin{proof} Denote by $\gamma$ the gaussian probability measure on  $\R^{2d}$ from Bargmann's representation (\ref{Bargman}). Using Theorem \ref{barg} we obtain
	\[
	\left\Vert P^{s}\right\Vert _{a}^{2}
	= \left\Vert P\right\Vert _{L^{2s}(\mathbb{R}^{2d}, \gamma)}^{2s},
	\]
	so
	\[
	\lim_{s\rightarrow\infty}\sqrt[s]{\left\Vert P^{s}\right\Vert _{a}}
	=
	\lim_{s\rightarrow\infty} \left\Vert P\right\Vert _{L^{2s}(\mathbb{R}^{2d}, \gamma)}
	=
	\left\Vert P\right\Vert _{L^{\infty}(\mathbb{R}^{2d}, \gamma)}
	=
	\infty
	\]
	since $P$ is not constant.
\end{proof}

\smallskip \noindent \textbf{Data availability} Data sharing is not
applicable to this article since no data sets were generated or analyzed.

\smallskip \noindent \textbf{Declarations}

\smallskip \noindent \textbf{Conflict of interest} The authors declare that
they have no competing interests.

\smallskip \noindent \textbf{Funding}  \thanks{The first and second named authors were partially supported by Grants PID2019-106870GB-I00 and PID2022-138916NB-I00 (respectively) of the MICINN of Spain,  by ICMAT Severo Ochoa project
	CEX2019-000904-S (MICINN), and by
	the Madrid Government (Comunidad de Madrid - Spain)  V PRICIT (Regional Programme of Research and Technological Innovation), 2022-2024.}

\end{document}